\documentclass[a4paper,leqno]{amsart}

\usepackage{latexsym}
\usepackage[english]{babel}
\usepackage{fancyhdr}
\usepackage[mathscr]{eucal}
\usepackage{amsmath}
\usepackage{mathrsfs}
\usepackage{mathtools}
\usepackage{amsthm}
\usepackage{amsfonts}
\usepackage{amssymb}
\usepackage{amscd}
\usepackage{bbm}
\usepackage{graphicx}
\usepackage{subcaption}
\usepackage{graphics}
\usepackage{latexsym}
\usepackage{color}
\usepackage{enumerate}

\usepackage{pifont}
\usepackage{booktabs} 

\newcommand{\K}[2]{\mathcal{K}\left(#1,#2\right)}

\newcommand{\Kc}[2]{\overline{\K{#1}{#2}}}

\newcommand{\norm}[2]{\left \Vert #1 \right \Vert_{#2}}
\newcommand{\opnorm}[1]{\left \Vert #1 \right \Vert_{\mathrm{op}}}

\newcommand{\ran}{\mathrm{ran}}

\newcommand{\cH}{\mathcal{H}}

\newcommand{\1}{\mathbbm{1}}

\newcommand{\C}{\mathbb{C}}
\newcommand{\N}{\mathbb{N}}

\newcommand{\res}[2]{\mathcal{R}\left(#1, #2\right)} 

\newcommand{\rd}{\mathrm{d}}

\theoremstyle{plain}
\newtheorem{theorem}{Theorem}[section]
\newtheorem{lemma}[theorem]{Lemma}
\newtheorem{corollary}[theorem]{Corollary}
\newtheorem{proposition}[theorem]{Proposition}

\theoremstyle{definition}
\newtheorem{definition}[theorem]{Definition}

\newtheorem{remark}[theorem]{Remark}
\newtheorem{example}[theorem]{Example}

\numberwithin{equation}{section}

\title[Some results on Krylov solvability in Banach space]{Some results on Krylov solvability in Banach space and connections to spectral theory}

\author[N.~Caruso]{No\`{e} Angelo Caruso}
\address[N.~Caruso]{Dipartimento di Scienze Umane, L'Universit\`{a} degli Studi Link \\ Via del Casale di San Pio V, 44\\ 00165 Roma (ITALIA).}
\email{n.caruso@unilink.it}
\thanks{The author wishes to thank the Italian National Institute of Higher Mathematics (INdAM) and the Italian National Group for Mathematical Physics (GNFM) for their support.}
\keywords{Krylov solvability, Krylov subspace, ill-posed problems, inverse linear problems, unbounded operator, resolvent theory, spectral theory, Banach space}

\subjclass[2020]{47A10, 47B01, 46B20, 47A10, 65J22, 65J10}

\begin{document}

\begin{abstract}
This article contains some first steps of a general analysis in Banach space of the problem of Krylov solvability of the inverse linear problem. In contrast to the well-studied Hilbert space setting, the more general Banach space setting presents particular difficulties in creating the connection between Krylov solvability and structural properties of the Krylov subspace itself. At the centre of this is the fact that the closed Krylov subspace may not always have a topological complement. We also develop spectral tools in order to attack the problem using the resolvent operator and exploiting its holomorphic properties on the resolvent set.
\end{abstract}

\maketitle

\section{Introduction}\label{sec:intro}

Krylov subspace methods are a popular family of algorithms used to solve many inverse linear problems numerically (see, for example, the excellent monographs \cite{Liesen-Strakos-2003,Saad-2003_IterativeMethods,Engl-Hanke-Neubauer-1996,Hanke-ConjGrad-1995}). They have been described as one of the top 10 algorithms of the 20th century \cite{Dongarra-Sullivan-Best10-2000}, and have gained so much popularity owing to their speed and ease of implementation.

The inverse linear problem for a closed linear operator $A: X \to X$, posed on a general Banach space $X$, is formulated as follows  
\begin{equation}\label{eq:lin_inv}
Af = g\,,
\end{equation}
where $g \in X$ is a known datum vector, and $f \in X$ is a solution to the problem (if it exists). We say that the problem is: \emph{solvable} if $g \in \ran A$, i.e., there exists a solution $f$ to \eqref{eq:lin_inv}; \emph{well-defined} if, in addition to being solvable, the solution is unique, i.e., the operator $A$ is injective; and \emph{well-posed} if, in addition to being well-defined, the solution $f$ depends continuously on the datum vector $g$.

Krylov subspace methods search for a solution to \eqref{eq:lin_inv} in the closure of distinguished subspace known as the \emph{Krylov subspace} that is defined by
\begin{equation}\label{eg:Ksubspace_def}
\K{A}{g} := \mathrm{span}\{A^k g\,|\, k \in \N_0\}\,.
\end{equation}
We note that for the above definition to make sense, it is necessary that $g$ is in the domain of all positive integer powers of the operator $A$, i.e., $g \in C^{\infty}(A)$ and thus we say that $g$ is $A$-smooth. The advantage of searching for solutions to \eqref{eq:lin_inv} in the closure of the Krylov subspace is that an approximation can be made using finite linear combinations of the vectors $g,Ag,A^2g,\dots$ that are in principle easy to construct. When a solution(s) to the inverse linear problem \eqref{eq:lin_inv} exists in the closure of the Krylov subspace, we say the the inverse problem is \emph{Krylov solvable}. We use the phrase \emph{Krylov solvability} to refer to this phenomenon in a more informal sense.

Over many years investigations have taken place into the properties of Krylov subspace algorithms, particularly their rates of convergence. There has been a very large emphasis on analysing these algorithms in the context of finite dimensional spaces where the Krylov solvability is under full control--see for example the monographs \cite{Liesen-Strakos-2003,Saad-2003_IterativeMethods}. There is also an amount of literature, though less developed, that covers the analysis of these methods in the infinite dimensional Hilbert space setting, for example \cite{Engl-Hanke-Neubauer-1996,Hanke-ConjGrad-1995,Karush-1952,Nemirovskiy-Polyak-1985,Nemirovskiy-Polyak-1985-II,Nemirovski-1986,Daniel-1967,Campbell-Ipsen-1996-I,Campbell-Ipsen-1996-II}. These investigations are primarily concerned with analysing specific algorithms and specific operator types (e.g. self-adjoint, compact, etc). Our previous works \cite{C-KrylovNormal-2022,C-2025-KrylovStruct,CM-Nemi-unbdd-2019,CM-2019_ubddKrylov,CM-krylov-perturbation-2021,CMN-2018_Krylov-solvability-bdd,CM-2025-KrylovReview} as well as our recent monograph \cite{CM-book-Krylov-2022} has sought to expand the analysis in the infinite-dimensional Hilbert space setting and answer what we believe to be a critical question: \emph{under what circumstances does a solution to \eqref{eq:lin_inv} (should at least one exist) exist in the closure of the Krylov subspace $\Kc{A}{g}$?} In the affirmative case, we say that the inverse linear problem is \emph{Krylov solvable}. The advantage in establishing the Krylov solvability (or lack of) is that one may choose the appropriate numerical algorithm to find a solution(s) to \eqref{eq:lin_inv} before actually running it.

Following on from our previous works in the Hilbert space setting, it seems natural to begin the expansion of the analysis to the more general Banach space setting. In contrast to the Hilbert space setting, there is even less literature available regarding Krylov subspaces in Banach space, though some results are available (for example see \cite{GrimmGockler-2017-KrylovC0semigroups}). The Banach space setting also presents challenges that in general arise as a consequence of the lack of a scalar product, and therefore a less well-defined geometry.

This paper is split into two main sections. Section~\ref{sec:structural_props} reconsiders the structural and geometric properties of $\K{A}{g}$ and how they relate to Krylov solvability. In particular there are several examples and counterexamples that demonstrate that the Krylov subspace may not even have topological complement in its embedding Banach space $X$, and moreover that the presence or lack of Krylov solutions is not necessarily dependent on this. This is in stark contrast to the Hilbert space setting as originally revealed in several past works \cite{CM-book-Krylov-2022,CM-2019_ubddKrylov,CMN-2018_Krylov-solvability-bdd,C-2025-KrylovStruct,C-KrylovNormal-2022}. Section~\ref{sec:spectral_approach} then changes perspective of the problem by considering the resolvent formulation. Using resolvents presents particular advantages in terms of dealing with unbounded operators, and avoids structural questions altogether. Several interesting results that exploit a complex-analysis approach to the resolvent are presented therein.

\subsection{Notation}~

Throughout the article various standard notations will be used. We use the symbols $X$ and $\cH$ to refer to abstract Banach and Hilbert spaces respectively, each with respective norms $\norm{\cdot}{X}$ and $\norm{\cdot}{\cH}$. The spaces $\mathscr{B}(X)$ and $\mathscr{C}(X)$ refer to the spaces of operators from the space $X$ to itself that are bounded everywhere defined, and closed, respectively. We equip $\mathscr{B}(X)$ with the standard operator norm $\opnorm{\cdot}$ that makes it additionally a Banach space. Lastly, we use $\mathcal{D}(A)$ to denote the domain of an operator $A$.

\section{Comments on the structural and geometric properties of $\K{A}{g}$}\label{sec:structural_props}

We begin our investigation by re-analysing certain structural concepts of Krylov subspaces in Hilbert space that were first uncovered in the article \cite{CMN-2018_Krylov-solvability-bdd} and later expanded upon in \cite{CM-2019_ubddKrylov,C-2025-KrylovStruct}. Owing to the lack of a scalar product, and the fact that identification of the space $X$ with its dual space $X^*$ is in general not possible, we focus our attention on one particular structural property: the \emph{Krylov intersection}. Before delving into this topic, we first recall some preliminary concepts with regards to topological complements of closed subspaces (see, \cite{Brezis-FA-Sob-PDE,AlbiacKalton-2016_Banachbook}).

\begin{definition}\label{def:top_complement}
Let $X$ be a Banach space and $E \subset X$ be a closed subspace. We say that a subspace $G \subset X$ is a \emph{topological complement} (or simply a \emph{complement}) of the subspace $E$ if 
\begin{itemize}
	\item[(i)] $G$ is closed,
	\item[(ii)] $E \cap G = \{0\}$, and
	\item[(iii)] $E + G = X$.
\end{itemize}

Given $E \subset X$ a closed subspace, if such a subspace $G$ exists then we say that $E$ is a complemented subspace of $X$.
\end{definition}
We note that if $E \subset X$ is a complemented subspace with complement $G$, then every $x \in X$ may be written in the decomposed form $x = z + y$ for $z \in E$ and $y \in G$ unique. Moreover, the mappings $x \xmapsto{P_E} z$ and $x \xmapsto{P_G} y$ are linear, continuous projection operators, with  $\1 - P_E = P_G$ \cite{AlbiacKalton-2016_Banachbook,Brezis-FA-Sob-PDE}.

It is already known that any Banach space $X$ such that \emph{every} closed subspace is a complemented subspace, is actually a Hilbert space \cite{AlbiacKalton-2016_Banachbook}. Clearly, in a general Banach space (non-Hilbert) there exist closed subspaces that admit no complement. A good example of a subspace without complement is the subspace $c_0$ of all sequences that converge to $0$ in $\ell^{\infty}(\N)$ \cite{AlbiacKalton-2016_Banachbook,Whitley-1966mtoc0}.

In our study of the structural properties of the Krylov subspace we shall focus on one specific property, namely the \emph{Krylov intersection}, which was a crucial structural-geometric property in determining the Krylov solvability of a linear inverse problem in the Hilbert space setting. Before defining this property we begin with some background.

We consider $A: X \to X$ as a closed linear operator of domain $\mathcal{D}(A)$ on a Banach space $X$. As the operator is closed, it is well-known that the graph space $V = (\mathcal{D}(A), \norm{\cdot}{A})$ is itself a Banach space, where $\norm{\cdot}{A}$ is the graph norm
\[\Vert x \Vert_A := \Vert x \Vert_X + \Vert Ax \Vert_X\,, \quad \forall \, x \in \mathcal{D}(A)\,.\]
It is clear that when $\mathcal{D}(A) = X$, by the closed graph theorem, $A$ is a bounded, everywhere defined linear operator on $X$, and moreover the graph norm and ambient norms are equivalent.

$A : V \to X$ is always a \emph{bounded} linear operator from the Banach space $V$ to the Banach space $X$. Therefore, for the rest of this section, we generically consider $A : V \to X$ and the inverse linear problem
\begin{equation}\label{eq:lin_inv_V}
A f = g\,,
\end{equation}
for $g \in C^\infty(A)$ a known $A$-smooth vector. When we speak of the \emph{closure} of the Krylov subspace, it can be the closure using the Banach space topology of $X$, denoted $\Kc{A}{g}$; or the closure in $V$, which shall be noted $\Kc{A}{g}^V$. We use this notation to avoid confusion arising over the topology of closure.

\begin{definition}\label{def:Krylov_int}
Let $A: X \to X$ be a closed linear operator on a Banach space $X$ with graph space $V=(\mathcal{D}(A),\norm{\cdot}{A})$, and suppose that $g \in C^\infty(A)$. Furthermore, suppose that $\Kc{A}{g}^V$ is a complemented subspace in $V$ with topological complement $G \subset V$. Then the \emph{Krylov intersection} is defined as follows.
\begin{equation}\label{eq:def_Krylov_int}
\mathcal{I}(A,g) := \Kc{A}{g} \cap A(G)\,.
\end{equation}
\end{definition}

We note that the Krylov intersection has already been extensively discussed in the Hilbert space setting \cite{CMN-2018_Krylov-solvability-bdd,CM-2019_ubddKrylov,C-2025-KrylovStruct}. Below we collect the pertinent results that allow us to establish its connection to the Krylov solvability, suitably modified for the Banach space setting.

First, we begin with a rather obvious result.
\begin{proposition}[Proposition~2.2 \cite{C-2025-KrylovStruct}]\label{prop:KrylovStruct_2.2}
Let $A:X \to X$ be a closed linear operator on the Banach space $X$ and $g \in C^\infty(A)$. Then 
\begin{equation}\label{eq:prop}
A \Kc{A}{g}^V \subset \Kc{A}{g}\,.
\end{equation}
\end{proposition}
\begin{proof}
This is a simple consequence of the continuity of $A$ as an operator from $V$ to $X$ as well as the inclusion $A \K{A}{g} \subset \K{A}{g}$.
\end{proof}

\begin{lemma}[Lemma~2.3 \cite{C-2025-KrylovStruct}]\label{lem:KrylovStruct_2.3}
Let $A:X \to X$ be a closed linear operator on Banach space and let $0 \in \rho(A)$. Then, in addition to $A^{-1} \in \mathscr{B}(X)$, we have $A^{-1} \in \mathscr{B}(X,V)$. The converse also holds, i.e., if $A^{-1}$ exists and $A^{-1} \in \mathscr{B}(V,X)$, then $0 \in \rho(A)$.
\end{lemma}
\begin{proof}
This lemma is a trivially modified \cite[Prop.~2.2]{C-2025-KrylovStruct}, and the proof therein follows with $X$ a Banach space in place of a Hilbert space.
\end{proof}

Similarly, the proof of Proposition~\ref{prop:KrylovStruct_2.4} below carries over from \cite[Prop.~2.4]{C-2025-KrylovStruct}.
\begin{proposition}[Proposition~2.4 \cite{C-2025-KrylovStruct}]\label{prop:KrylovStruct_2.4}
Let $A:X \to X$ be a closed and injective operator on a Banach space $X$, and let $g \in C^\infty(A) \cap \ran A$. Let $f \in \mathcal{D}(A)$ be the solution to $Af = g$.
\begin{itemize}
	\item[(i)] If $f \in \Kc{A}{g}^V$ then $A \Kc{A}{g}^V$ is dense in $\Kc{A}{g}$.
	\item[(ii)] Assume further that $0 \in \rho(A)$. Then $f \in \Kc{A}{g}^V$ if and only if $A\Kc{A}{g}^V$ is dense in $\Kc{A}{g}$.
\end{itemize}
\end{proposition}

We are now in a position to formulate and prove the main result linking Krylov solvability to the Krylov intersection.
\begin{theorem}\label{th:Kint_Ksolv}
Let $A:X \to X$ be a closed linear operator on a Banach space $X$, let $g \in C^\infty(A) \cap \ran A$. Furthermore, suppose that there exists a topological complement in $V$, $G \subset V$, to the closed Krylov subspace $\Kc{A}{g}^V$ (i.e., the Krylov intersection $\mathcal{I}(A,g)$ is indeed defined).
\begin{itemize}
	\item[(i)] If the Krylov intersection $\mathcal{I}(A,g) = \{0\}$ then there exists a solution $f \in \mathcal{D}(A)$ to inverse linear problem $A f = g$ such that $f \in \Kc{A}{g}^V$.
	\item[(ii)] If in addition $0 \in \rho(A)$, then the solution $f$ to the inverse linear problem $Af=g$, $f \in \Kc{A}{g}^V$ if and only if $\mathcal{I}(A,g) = \{0\}$.
\end{itemize}
\end{theorem}
\begin{proof}
The proof here is a modification of the proof of \cite[Th.~2.6]{C-2025-KrylovStruct}.

We begin with part (i). Suppose that $\mathcal{I}(A,g) = \{0\}$. Due to the fact that $\Kc{A}{g}^V$ is complemented in $V$, there exists two continuous linear projection operators $P_K : V \to V$ and $P_G : V \to V$, the projection onto $\Kc{A}{g}^V$ and its complement $G$ respectively. We note that each $x \in V$ may be written uniquely as $x = P_K x + P_G x$, and therefore that $P_G = \1 - P_K$.

Let $f$ be a solution to $Af=g$. Then $f = P_K f + P_G f$. Thus $Af = AP_K f + AP_G f$ and as $Af = g \in \Kc{A}{g}^V \subset \Kc{A}{g}$ and $A P_K f \in \Kc{A}{g}$ (Proposition~\ref{prop:KrylovStruct_2.2}) we see that $A P_G f \in \Kc{A}{g}$. Yet $A P_G f \in A(G)$, and as $\mathcal{I}(A,g) = \{0\}$ we have that $A P_G f = 0$. Therefore, $A P_K f = g$, and so $\tilde{f} = P_K f \in \Kc{A}{g}^V$ is a Krylov solution to the inverse linear problem.

We now prove the `only if' implication of part (ii). Suppose $0 \in \rho(A)$ and that the solution $f$ to $Af = g$ is such that $f \in \Kc{A}{g}^V$. Let $w \in \mathcal{I}(A,g)$ so that $w \in \Kc{A}{g}$ and there exists a unique vector $v \in G$ such that $Av=w \in A(G)$.

By Proposition~\ref{prop:KrylovStruct_2.4} $A \Kc{A}{g}^V$ is dense in $\Kc{A}{g}$ and therefore there exists a sequence $(x_n)_{n \in \N} \subset \Kc{A}{g}^V$ such that $\Vert Ax_n - w \Vert_X \xrightarrow{n \to +\infty } 0$. Thus
\[\Vert x_n - v \Vert_A = \Vert A^{-1}(Ax_n - w) \Vert_A \leqslant \Vert A^{-1} \Vert_{\mathscr{B}(V,X)} \Vert Ax_n - w \Vert_X \xrightarrow{n \to +\infty} 0\,,\]
where we have used Lemma~\ref{lem:KrylovStruct_2.3}. Thus $v \in \Kc{A}{g}^V$, and as $G$ is a topological complement to $\Kc{A}{g}^V$, we have $\Kc{A}{g}^V \cap G = \{0\}$. Thus $v=0$ implying that $w = 0$.
\end{proof}

\begin{remark}\label{rem:Kint_top_complement}
In the proof of Theorem~\ref{th:Kint_Ksolv} we see that the non-trivial assumption that $\Kc{A}{g}^V$ is complemented in $V$ is crucial in establishing the existence of the continuous projection operators as well as the fact that $\Kc{A}{g}^V \cap G = \{0\}$. This is all taken for granted when $X$ is a Hilbert space, and in fact in the original version of Theorem~\ref{th:Kint_Ksolv} \cite[Theorem~2.6]{C-2025-KrylovStruct} the extra assumption is not even stated.
\end{remark}

\begin{remark}\label{rem:K_reducibility}
As in the study \cite{CMN-2018_Krylov-solvability-bdd}, under the conditions that $A$ is bounded everywhere defined (so that $\norm{\cdot}{A}$ and $\norm{\cdot}{X}$ are equivalent) we have that 
\[A \Kc{A}{g} \subset \Kc{A}{g}\,.\]
Furthermore, if $\Kc{A}{g} \subset X$ is complemented with topological complement $G \subset X$, and if also
\[A(G) \subset G\,,\]
we say that the operator $A$ is $\K{A}{g}$-reduced. In following with \cite{CMN-2018_Krylov-solvability-bdd} the phenomenon of Krylov reducibility gives way to results on the Krylov solvability of the linear inverse problem. Yet, seeing as the $\K{A}{g}$-reducibility of an operator $A$ implies the triviality of the Krylov intersection, we shall not explore further this phenomenon.
\end{remark}

\subsection{The lack of a topological complement to $\Kc{A}{g}^V$}~

The fact that $\Kc{A}{g}^V$ is complemented in $V$ is \emph{not} necessarily a given. Indeed it is a non-trivial assumption, as we shall illustrate with several examples. Yet, we note that such an assumption is not necessary to guarantee the Krylov solvability of an inverse linear problem.

\begin{example}\label{eg:no_Krylov_complement_notSolvable}
We consider as our Banach space $X$ the space of all bounded sequences $\ell^\infty(\N)$. Immediately we have that, as $\ell^\infty(\N)$ is not separable, any Krylov subspace can never be dense in $\ell^\infty(\N)$. It is well-known that $c_0 \subset \ell^\infty(\N)$, the space of all sequences that converge to zero, i.e.
\[c_0 = \{ (\xi_n)_{n \in \N} \in \ell^\infty(\N) \,|\, \xi_n \xrightarrow{n \to +\infty} 0\}\,,\]
is \emph{not} complemented in $\ell^\infty(\N)$ \cite{AlbiacKalton-2016_Banachbook}, and thus there does \emph{not} exist a continuous projection operator $P_0:\ell^\infty(\N) \to \ell^\infty(\N)$ with $P_0(\ell^\infty(\N)) = c_0$.

It is also well-known that $c_{00}$, the space of all sequences with \emph{finite} support, i.e.,
\[c_{00} = \{ (\xi_n)_{n \in \N} \,|\, \xi_n \neq 0 \text{ for finitely many } n\}\,,\]
is \emph{dense} in $c_0$.

Let $A:\ell^\infty(\N) \to \ell^\infty(\N)$ be the forward shift given by the action $e_n \mapsto e_{n+1}$ on the $n$-th canonical vector $e_n$. Thus $A$ is a bounded everywhere defined linear operator on $\ell^\infty(\N)$.

If we suppose that $g = e_1$, then when we consider finite linear combinations of vectors $A^kg$, for $k \in \N_0$, we get $\K{A}{g} = c_{00}$. Therefore $\Kc{A}{g} = c_0 \subset \ell^\infty(\N)$, thus demonstrating the possibility to generate a non complementable closed Krylov subspace. In this case we have that $g \notin \ran A$ which makes equation \eqref{eq:lin_inv} itself not solvable.

As already noted elsewhere on Hilbert space \cite{CMN-2018_Krylov-solvability-bdd,CM-book-Krylov-2022}, it is clear that in the case where $g \in \ran A$, the inverse linear problem $Af = g$ is \emph{not} Krylov solvable anyway. Indeed this applies also for the Banach space $\ell^\infty(\N)$ as for any $g = (\xi_n)_{n \in \N}$ such that $g \in \ran A$, it must be that $\xi_1 = 0$. Suppose, without loss of generality, that $\xi_2 \neq 0$. We have that the unique solution $f$ to \eqref{eq:lin_inv} is
\[f \equiv (\beta_n)_{n \in \N}\,, \quad \beta_n = \xi_{n +1} \quad \forall\, n \in \N\,.\]
It is clear that 
\[\Kc{A}{g} \subset \overline{\mathrm{span}\{e_n \, |\, n \in \N\,, \, n \geqslant 2\}}\,,\]
and thus, as $\beta_1 \neq 0$, we have that $f \notin \Kc{A}{g}$. Indeed, $(\alpha,0,0,\dots) \notin \Kc{A}{g}$ for any $\alpha \neq 0$.
\end{example}

We now note by way of another example that the lack of topological complement to $\Kc{A}{g}$ is not a phenomenon necessarily related to the lack of Krylov solvability of the inverse linear problem, or even the lack of solvability itself.

\begin{example}\label{eg:no_Krylov_complement_Ksolvable}
We now present an example where we have a solvable linear inverse problem that is also Krylov solvable, yet the closed Krylov subspace has no topological complement.

Let $X = \ell^\infty(\N)$ and $A: \ell^\infty(\N) \to \ell^\infty(\N)$ have the action $e_n \mapsto \sigma_n e_n$ for $e_n$ the $n$-th canonical vector. We consider the weights 
\[\sigma_n = \frac{1}{\sqrt{n}}\,,\quad \forall\, n \in \N\,,\]
so that the operator $A$ is continuous and everywhere defined, injective, and we consider the vector 
\[g \equiv \bigg( \frac{1}{n}\bigg)_{n \in \N}\,,\]
so that $g \in \ran A$. The unique solution to the inverse linear problem \eqref{eq:lin_inv} is 
\[f \equiv \bigg(\frac{1}{\sqrt{n}}\bigg)_{n \in \N}\,,\]
and we note that both $g$ and $f$ are in $c_0$. 

We claim that $\Kc{A}{g} = c_0$ and therefore the Krylov subspace has no topological complement, yet the inverse linear problem is Krylov solvable. Indeed, we consider the vector $A^kg$ for $k \in \N_0$,
\[A^kg \equiv \bigg(\frac{1}{n} \cdot \frac{1}{n^{\frac{k}{2}}} \bigg)_{n \in \N} = \bigg(\frac{1}{n^{1+ \frac{k}{2}}}\bigg)_{n \in \N}\,.\]
Therefore the vectors $A^kg$ actually defines the collection of functions on $\N$ 
\[\varphi_k : \N \to \C\,, \quad n \mapsto \frac{1}{n^{1+\frac{k}{2}}}\,,\]
for $k \in \N_0$, and $A^kg \equiv (\varphi_k(n))_{n \in \N}$ for all $k \in \N_0$. Thus we can consider the Krylov subspace $\K{A}{g}$ as being represented by the span of these functions. Working directly with this span of functions
\[\Phi = \mathrm{span}\{\varphi_k \, | \, k \in \N_0\}\,,\]
we note that it satisfies the conditions of the Stone-Weierstrass theorem, i.e., $\Phi$
	\begin{itemize}
		\item[(i)] separates points, i.e., for every two points $n,m \in \N$ there exists a function $\psi \in \Phi$ such that $\psi(n) \neq \psi(m)$ when $n \neq m$, 
		\item[(ii)] vanishes nowhere, i.e., for all $n \in \N$ there exists a function $\psi \in \Phi$ such that $\psi(n) \neq 0$, and
		\item[(iii)] forms an involutive subalgebra of functions on $\N$ that decay to zero, i.e., forms a subalgebra of $c_0$. 
	\end{itemize}
(The involution operation here is clearly the complex conjugate operation.) Indeed, $\Phi$ separates points: consider the function $\varphi_0$, then 
\[\varphi_0(n) = \frac{1}{n}\,, \quad \varphi_0(m) = \frac{1}{m}\,,\]
and it is clear that $\varphi_0(n) \neq \varphi_0(m)$ when $n \neq m$. It is very easy to see that $\Phi$ vanishes nowhere, indeed it is the case that for any $n \in \N$, $\varphi_k(n) \neq 0$ for all $k \in \N$. Lastly, to show that $\Phi$ forms an involutive subalgebra, it is enough to show that the sum and product of two spanning functions are themselves spanning functions. The fact that the spanning functions are real valued already gives us the fact that $\Phi$ is closed under the involution operation. By the definition of the span it is clear that $\alpha \varphi_k + \beta \varphi_l \in \Phi$ for $k,l \in\N$ and for $\alpha,\beta \in \C$. Thus, consider $\tilde{\varphi} = \varphi_k \cdot \varphi_l$. We see that
\[\tilde{\varphi}(n) = \varphi_k(n) \cdot \varphi_l(n) = \frac{1}{n^{2+\frac{k+l}{2}}}\,,\]
and thus 
\[\varphi_{2+k+l}(n) = \frac{1}{n^{1+ \frac{2+k+l}{2}}} = \frac{1}{n^{2+\frac{k+l}{2}}} = \tilde{\varphi}(n)\,,\]
which means that $\varphi_k \cdot \varphi_l \in \Phi$. Therefore, $\Phi$ forms a subalgebra of functions on $\N$ that decay to zero. This implies that $\Phi$ forms an involutive subalgebra of $c_0$ that separates points of $\N$ and vanishes nowhere on $\N$. 

Therefore, by the Stone-Weierstrass theorem for the locally compact space $\N$, we see that the closure of $\Phi$, and thus also $\K{A}{g}$, in the supremum norm is precisely all of $c_0$. Thus, the inverse linear problem is Krylov solvable, and the closed Krylov subspace has no topological complement in $\ell^\infty(\N)$ thereby proving our claim.
\end{example}

\begin{example}\label{eg:no_Krylov_complement_notKsolvable}
In this follow-up of Example~\ref{eg:no_Krylov_complement_Ksolvable} we use the same Banach space $\ell^\infty(\N)$ and operator $A:\ell^\infty(\N) \to \ell^\infty(\N)$,
\[e_n \mapsto \frac{1}{\sqrt{n}} e_n\,.\]
Now we consider the vector $g$ as
\[g \equiv \bigg(\frac{1}{\sqrt{n}}\bigg)_{n \in \N}\,,\]
so that the unique solution to the inverse linear problem $Af=g$ is $f \equiv (1)_{n \in \N}$. By the same reasoning as in Example~\ref{eg:no_Krylov_complement_Ksolvable}, using the Stone-Weierstrass theorem we see that $\Kc{A}{g} = c_0$. As the unique solution is a vector of ones, it clearly is not in $c_0$.

Thus we have an example of a situation in which $\Kc{A}{g}$ has no topological complement and the linear inverse problem, though solvable, is not \emph{Krylov} solvable.
\end{example}

\subsection{The existence of a topological complement to $\Kc{A}{g}^V$}~

We now move on to show some examples where we have the existence of a topological complement to $\Kc{A}{g}$ ($A$ a bounded operator) in a Banach space $X$, that is \emph{not} a Hilbert space. First, we state some preliminary results.

\begin{proposition}\label{prop:weighted_ellp_space}
Consider the space $L^p(\N, \mu)$ for $p \in [1,+\infty)$ and $\mu$ the weighted counting measure given by
\begin{equation}
\mu(\Omega) := \int_\Omega \varphi(n) \, \rd \nu \quad (\text{i.e., } \rd \mu = \varphi \, \rd \nu)\,,
\end{equation}
for $\Omega \subset \N$ Borel, $\nu$ the counting measure on $\N$, and $\varphi$ a strictly positive, $\nu$-measurable bounded function on $\N$ (for example, $\varphi(n) = \exp(-n)$). If $S \subset \N$, then the space
\[M := \big\{ \chi_S(n)f(n) \, | \, f \in L^p(\N,\mu) \big\}\,,\]
is complemented in $L^p(\N,\mu)$, i.e., it is a closed linear subspace with a topological complement $G$.
\end{proposition}

\begin{remark}\label{rem:prop_weighted_ellp_space}
We note that the space $L^p(\N,\mu)$ is actually an $\ell^p(\N)$ space with the weighted norm
\[\norm{x}{L^p} = \bigg(\sum_{n \in \N} \varphi(n) \vert x_n \vert^p\bigg)^{\frac{1}{p}}\,.\]
Of course for $\varphi \equiv 1$ we recover the traditional $\ell^p(\N)$ space. 
\end{remark}

\begin{proof}[Proof of Proposition~\ref{prop:weighted_ellp_space}]
Let $S' = \N \setminus S$, and we know that $\nu$ and $\mu$ are positive Borel regular, $\sigma$-finite measures. First we shall show that $M$ is linear. Indeed, let $f$ and $g$ be functions in $M$, and $\alpha,\beta \in \C$, so that $\chi_S f = f$, $\chi_S g = g$,u and we see that $\alpha f + \beta g = (\alpha f+\beta g)\chi_S \in M$.

Secondly we show that $M$ is closed. Indeed, suppose that $(f_m)_{m \in \N}$ is a Cauchy sequence in $M$, i.e., it converges to some $f \in L^p(\N,\mu)$. Thus, in particular $f_m = \chi_S f_m \in M$ for all $m \in \N$, and so
\[\int_\N \vert \chi_S(n)f_m(n) - f(n) \vert^p \, \rd \mu \xrightarrow{m \to +\infty} 0\,.\]
Therefore there exists a subsequence $(\chi_S(n)f_{m_k}(n))_{k\in\N}$ that converges pointwise a.e. (and thus everywhere on $\N$) to $f(n)$, i.e.,
\[\chi_S(n)f_{m_k}(n) \xrightarrow{k \to +\infty} f(n) \,, \quad \forall\,n \in \N\,.\]
It is clear that $f(n) = 0$ for all $n \in S'$, and so $f = \chi_S f$ from which we see that $f \in M$. Therefore $M$ is closed.

Now we consider the space 
\[G = \big\{\chi_{S'} f \, | \, f \in L^p(\N,\mu) \big\}\,,\]
that we see is also linear and closed by the arguments above. Let $f \in M \cap G$, so that
\[\chi_S(n) f(n) = f(n) = \chi_{S'}(n)f(n)\,, \quad \forall \, n \in \N\,,\]
which implies that $f \equiv 0$ as $S \cap S' = \emptyset$. Thus $M \cap G = \{0\}$. 

As $\chi_S + \chi_{S'} \equiv 1$, for any $f \in L^p(\N,\mu)$, we have that 
\[f = \chi_S f + \chi_{S'} f\,,\]
and this implies that $L^p(\N,\mu) \subset M + G$. The reverse inclusion is obvious, and therefore $M + G = L^p(\N,\mu)$.

Thus, by the definition of topological complement, we see that $M$ is indeed complemented in $L^p(\N, \mu)$ and that $G$ is its topological complement.
\end{proof}

\begin{remark}\label{rem:prop_weighted_ellp}
We note that under the conditions of Proposition~\ref{prop:weighted_ellp_space} the projection operators $P_M$ and $P_G$ are precisely the multiplication by the characteristic functions $\chi_S$ and $\chi_{S'}$ respectively, i.e.,
\[P_M:L^p(\N,\mu) \to L^p(\N,\mu)\,, \quad f \mapsto \chi_S f\,,\]
and similarly for $P_G$. These are clearly idempotent, bounded, everywhere defined operators on $L^p(\N,\mu)$, and indeed do define their respective projections onto $M$ and $G$.
\end{remark}

\begin{example}\label{eg:Krylov_solvable_topological_complement}
We consider the bounded, everywhere defined injective operator $A:L^p(\N,\mu) \to L^p(\N,\mu)$, $f(n) \mapsto \frac{1}{n}f(n)$, $p\in[1,+\infty)$ for the space $L^p(\N,\mu)$ as in Proposition~\ref{prop:weighted_ellp_space} for $\varphi(n) = \exp(-n)$. In this case we therefore have $\mu(\N) < + \infty$.

Let 
\[f(n) = \begin{cases} 1\,, & n \text{ even} \\ 0\,, & n \text{ odd} \end{cases}\,,\]
so that $f \in L^p(\N,\mu)$ and moreover for $S = \{2n \,|\, n \in \N\}$ we have that $f \in M$ (for $M$ as defined in Proposition~\ref{prop:weighted_ellp_space}). Then the vector $g \in \ran A$ such that $Af = g$ is also in $M$, where
\[g(n) = \begin{cases} \frac{1}{n}\,, & n \text{ even} \\ 0\,, & n \text{ odd} \end{cases}\,.\]
Now we consider the function $(A^kg)(n)$ for $k \in \N_0$,
\[(A^kg)(n) = \begin{cases} \frac{1}{n^{1+k}}\,, & n \text{ even} \\ 0\,, & n \text{ odd} \end{cases}\,, \quad \forall \, k \in \N_0\,,\]
so that $A^kg \in M$ as well. Thus it is clear that $\Kc{A}{g} \subset M$, and we claim that $\K{A}{g}$ is itself \emph{dense} in $M$. Indeed, we claim that the functions in $\K{A}{g}$ satisfy the conditions of the Stone-Weierstrass theorem on $S$, i.e., they
\begin{itemize}
		\item[(i)] separate points, i.e., for every two points $n,m \in S$ there exists a function $h \in \K{A}{g}$ such that $h(n) \neq h(m)$ when $n \neq m$, 
		\item[(ii)] vanish nowhere, i.e., for all $n \in S$ there exists a function $h \in \K{A}{g}$ such that $h(n) \neq 0$, and
		\item[(iii)] form an involutive subalgebra of functions on $S$ that decay to zero.
	\end{itemize}
Indeed, the function $g$ itself separates points on $S$ thus proving condition (i), and it also satisfies condition (ii). Thus it remains only to show condition (iii), for which it is necessary only to show that $(A^kg)(n) \cdot (A^mg)(n) \in \K{A}{g}$, as certainly all (finite) linear combinations of $(A^kg)(n)$ are in the Krylov subspace $\K{A}{g}$. In fact
\begin{align*}
(A^kg)(n)\cdot(A^mg)(n) & = \begin{cases}\frac{1}{n^{2+k+m}}\,, & n \text{ even} \\ 0 \,, & n \text{ odd} \end{cases}\,,\\
 & = (A^{1+k+m}g)(n) \in \K{A}{g}\,.
\end{align*}
Therefore $\K{A}{g}$ is an involutive subalgebra of $M$. Applying the Stone-Weierstrass theorem, the closure of $\K{A}{g}$ in the \emph{supremum} norm is the set of all bounded functions with support on $S$ that vanish at infinity, i.e., $C_0^S(\N)$. It is clear that $C_0^S(\N) \subset M$. We may re-write $M$ as the space
\[M = L^p(\N, \mu')\,, \quad \rd\mu' = \chi_S \,\rd \mu\,,\] 
so that by standard approximation theorems (e.g., \cite[Th.~3.14]{Rudin-realcomplexanalysis}) $C_0^S(\N)$ is dense in $L^p(\N,\mu')$, i.e., it is dense in $M$. This implies that $\K{A}{g}$ is dense in $M$ in the $L^p(\N,\mu)$ norm. Indeed, given any function $\psi \in M = L^p(\N,\mu') \subset L^p(\N,\mu)$, for $\varepsilon > 0$ we may approximate it with a function $\phi_\varepsilon \in C_0^S(\N)$ in supremum norm so that
\[\Vert \psi - \phi_\varepsilon \Vert_{L^p(\mu)} < \varepsilon\,,\]
and therefore, given that $\K{A}{g}$ is dense in $C_0^S(\N)$ in supremum norm, there exists a polynomial $p_\varepsilon$ such that $\Vert p_\varepsilon(A)g - \phi \Vert_\infty < \varepsilon$. As such
\begin{align*}
\Vert p_\varepsilon(A)g - \psi \Vert_{L^p(\mu)} & \leqslant \Vert p_\varepsilon(A)g - \phi_\varepsilon \Vert_{L^p(\mu)} + \Vert \phi_\varepsilon - \psi \Vert_{L^p(\mu)} \,, \\
 & < \Vert p_\varepsilon(A)g - \phi_\varepsilon \Vert_{L^p(\mu)} + \varepsilon \,, \\
 & = \bigg( \int_\N \vert p_\varepsilon(A)g - \phi_\varepsilon \vert^p \, \rd \mu \bigg)^{\frac{1}{p}} + \varepsilon \,, \\
 & \leqslant \Vert p_\varepsilon(A)g - \phi_\varepsilon \Vert_\infty \mu(\N)^{\frac{1}{p}} + \varepsilon \,,\\
 & < (1+ \mu(\N)^{\frac{1}{p}})\varepsilon\,.
\end{align*} 
As $\mu(\N) < +\infty$ the above inequality vanishes for $\varepsilon \to 0$, and thus $\Kc{A}{g} = M$. As such $\Kc{A}{g}$ has a topological complement (Proposition~\ref{prop:weighted_ellp_space}) and as $f \in M$ we have that the inverse linear problem is also Krylov solvable.
\end{example}

\begin{remark}
In fact, an additional result of Example~\ref{eg:Krylov_solvable_topological_complement} is that the operator $A$ is actually $\K{A}{g}$-reduced. For the topological complement $G$ to $\Kc{A}{g}$ (as defined in the proof of Proposition~\ref{prop:weighted_ellp_space}), it so happens that $A(G) \subset G$. Indeed, for $f \in G$ we have that $\chi_{S'} f = f$ so that $(Af)(n) = \frac{1}{n} \chi_{S'}(n)f(n) = \chi_{S'}(n)\frac{1}{n}f(n) \in G$. Thus the phenomenon of Krylov reducibility is not something unique to the Hilbert space setting.
\end{remark}

\begin{example}\label{eg:non_Krylov_solvable_topological_complement}
We may now consider another bounded, everywhere defined, injective operator $A:\ell^p(\N) \to \ell^p(\N)$, $e_n \mapsto e_{n+2}$ (another type of forward shift), $p\in[1,+\infty)$. We note that from Remark~\ref{rem:prop_weighted_ellp_space} $\ell^p(\N)$ is the same space as $L^p(\N,\mu)$ from Proposition~\ref{prop:weighted_ellp_space} for $\varphi \equiv 1$. Therefore, the space $M = \overline{\mathrm{span}\{e_{2n} \, | \, n \in \N\,, n \geqslant 2\}}$ has topological complement 
\[G = \overline{\mathrm{span}\{e_{2n+1}\,|\,n \geqslant 0\} + \mathrm{span}\{e_2\}}\,.\]

Let $f = e_2$, so that we generate the inverse linear problem $Af = g = e_4$ where we see that 
\[\K{A}{g} = \mathrm{span}\{e_{2n} \, | \, n \in \N\,, n \geqslant 2\}\,,\]
which is dense in $M$. Thus by Proposition~\ref{prop:weighted_ellp_space} $\Kc{A}{g}$ has topological complement $G$.

It is evident that $f \notin \Kc{A}{g}$ and therefore, though we have the existence of a topological complement of $\Kc{A}{g}$, this inverse linear problem is \emph{not} Krylov solvable.
\end{example}

\begin{remark}\label{rem:Krylov_complement}
The previous examples~\ref{eg:no_Krylov_complement_notSolvable}, \ref{eg:no_Krylov_complement_notKsolvable}, \ref{eg:no_Krylov_complement_Ksolvable}, \ref{eg:Krylov_solvable_topological_complement}, and \ref{eg:non_Krylov_solvable_topological_complement} together show that, in general, there is no relation between the Krylov solvability of an inverse linear problem and the existence of a topological complement to the closed Krylov subspace. Therefore the assumption of the existence of a topological complement is indeed a necessary one to show Krylov solvability by the purely projective/geometric techniques in this section. Moreover the absence of a topological complement to the closed Krylov subspace does not imply that the inverse linear problem is not Krylov solvable. Indeed in these cases we must switch to other tools available, namely spectral and perturbative.
\end{remark}

\section{A spectral approach to the problem}\label{sec:spectral_approach}
Spectral methods provide a powerful alternative to attacking the Krylov solvability problem. They are not subject to the extra assumptions that are needed for structural methods, namely the existence of a topological complement to $\Kc{A}{g}$. Indeed, spectral techniques relating to Krylov solvability have been investigated elsewhere \cite{CMN-2018_Krylov-solvability-bdd,C-2025-KrylovStruct,CM-Nemi-unbdd-2019,CM-book-Krylov-2022} where they have proven to be very useful.

In this section rather than use directly the Dunford-Schwartz functional calculus for bounded operators in establishing Krylov solvability, we choose to look at a different facet by considering the resolvent set and resolvent operator. This formulation has the advantage in that it encompasses the whole class of closed operators on a Banach space $X$, and moreover may prove to be useful to analyse Krylov solvability under the effects of perturbations.

We begin by listing the main results and their corollaries, along with appropriate remarks.

\begin{theorem}\label{th:poly_resolvent_opnorm}
Let $A:X \to X$ be an everywhere defined and bounded operator, and let $\mathbb{P}(A)$ be the set of operators
\[\mathbb{P}(A) := \{ p(A) \,|\,p \text{ a polynomial}\} \subset \mathscr{B}(X)\,.\]
Consider any $\zeta_0 > \mathrm{spr} A$. Then the resolvent operator $\res{\zeta_0}{A}$ is such that
\[\res{\zeta_0}{A} \in \overline{\mathbb{P}(A)}^{\opnorm{\cdot}}\,.\]
Moreover, $\res{\zeta}{A} \in \overline{\mathbb{P}(A)}^{\opnorm{\cdot}}$ for all $\zeta \in \rho(A,\zeta_0)$, where $\rho(A,\zeta_0)$ is the connected part of $\rho(A)$ containing $\zeta_0$.
\end{theorem}

\begin{remark}\label{rem:poly_resolvent}
It is interesting to note that from Theorem~\ref{th:poly_resolvent_opnorm} we see that the operator functions
\[(\zeta\1 - A)^{-1}\]
are actually approximable by polynomials in $A$ in operator norm when $\zeta > \mathrm{spr}A$. The significance of this is that certain rational Krylov subspaces that use the resolvent for $\zeta > \mathrm{spr}A$ are actually a subset of the Krylov subspace itself $\Kc{A}{g}$ (both for the same $g \in X$). Therefore, the question of Krylov solvability for the traditional polynomial Krylov subspace is relevant in these cases.
\end{remark}

\begin{corollary}\label{cor:k-class}
Let $A: X \to X$ be a bounded and everywhere defined operator of the $\mathscr{K}$-class (see \cite[Sect.~3.5]{CMN-2018_Krylov-solvability-bdd}). Then $A^{-1} \in \overline{\mathbb{P}(A)}^{\opnorm{\cdot}}$, and in particular, the inverse linear problem $Af = g$ is Krylov solvable for all $g \in X$.
\end{corollary}
\begin{proof}
As $A$ is of the $\mathscr{K}$-class of operators, we know that $0 \in \rho(A)$ and that $0$ is connected to at least one $\zeta_0 > \mathrm{spr} A$. Thus, by Theorem~\ref{th:poly_resolvent_opnorm} we have that $A^{-1} = \res{0}{A} \in \overline{\mathbb{P}(A)}^{\opnorm{\cdot}}$, and in particular there exists a polynomial sequence $(p_n)_{n \in \N}$ such that $p_n(A) \xrightarrow{\opnorm{\cdot}} A^{-1}$, whence $p_n(A)g \to A^{-1}g$ for all $g \in X$.
\end{proof}
\begin{remark}\label{rem:k-class}
In contrast to \cite{CMN-2018_Krylov-solvability-bdd}, in Corollary~\ref{cor:k-class} and the theorem upon which it is based, we do not need to invoke neither the Dunford-Schwartz functional calculus nor results from complex approximation theory. Indeed the result follows directly from the analytic properties of the resolvent (in $\zeta$).
\end{remark}

\begin{theorem}\label{th:resolvent_connected_Ksolv}
Consider $A : X \to X$ a closed operator and suppose that $g \in C^{\infty}(A)$. Let $\zeta_0 \in \rho(A)$, and let $\rho(A,\zeta_0)$ be the connected component of the resolvent containing $\zeta_0$. If $\res{\zeta_0}{A}g \in \Kc{A}{g}$ and $A\big(\Kc{A}{g}\cap\mathcal{D}(A)\big) \subset \Kc{A}{g}$, then $\res{\zeta}{A}g \in \Kc{A}{g}$ for all $\zeta \in \rho(A,\zeta_0)$.
\end{theorem}

\begin{remark}\label{rem:th_resolvent_connected_Krylov}
We note the condition $A\big(\Kc{A}{g}\cap\mathcal{D}(A)\big) \subset \Kc{A}{g}$ in Theorem~\ref{th:resolvent_connected_Ksolv} says that that the phenomenon of Krylov escape does \emph{not} occur (see \cite{CM-2019_ubddKrylov} for details and a counterexample), a phenomenon that indeed may occur in the unbounded operator setting. We further note that under the very natural assumption that the Krylov-core condition holds, i.e., $\Kc{A}{g}^{V} = \Kc{A}{g} \cap \mathcal{D}(A)$, Krylov escape does not occur. This condition is discussed in \cite{CM-2019_ubddKrylov} for operators in Hilbert space, and also more recently in \cite{C-2025-KrylovStruct} for self-adjoint operators.

Obviously, when $A$ is a bounded and everywhere defined operator, the Krylov core condition is satisfied and therefore there is no possibility of Krylov escape.
\end{remark}

\begin{remark}\label{rem:comparison_to_Hilbert}
We remark that in the case of Theorem~\ref{th:resolvent_connected_Ksolv} for $A$ a bounded everywhere defined operator in the \emph{Hilbert space} setting, if $0$ is in the resolvent set, then the presence of some $\zeta_0$ in the connected component $\rho(A,0)$ such that $\res{\zeta_0}{A}g \in \Kc{A}{g}$ is tantamount to the triviality of the Krylov intersection owing to the Krylov solvability (see \cite[Prop.~3.4]{CMN-2018_Krylov-solvability-bdd} for further details).
\end{remark}

\subsection{Proofs of Theorems~\ref{th:poly_resolvent_opnorm} and \ref{th:resolvent_connected_Ksolv}}~

We shall now prove both Theorems~\ref{th:poly_resolvent_opnorm} and \ref{th:resolvent_connected_Ksolv}. We begin first with the proof of Theorem~\ref{th:resolvent_connected_Ksolv}.

\begin{proof}[Proof of Theorem~\ref{th:resolvent_connected_Ksolv}]
We begin with the resolvent formula (see \cite[Sect.~5.2, Ch.~I]{Kato-perturbation})
\begin{equation}\label{eq:proof_thresolventconnectedKsolvI}
\res{\zeta}{A} = \sum_{n=0}^{\infty} (\zeta - \zeta_0)^n \res{\zeta_0}{A}^{n+1}\,, \quad \text{for}\,\,\, \vert \zeta - \zeta_0 \vert < \frac{1}{\mathrm{spr}\res{\zeta_0}{A}}\,, 
\end{equation}
which is convergent in $\opnorm{\cdot}$. It is also known (\cite[Prob.~6.16, Ch.~III]{Kato-perturbation}) that
\[\mathrm{spr}\res{\zeta_0}{A} = \frac{1}{\mathrm{dist}(\zeta_0,\sigma(A))}\,,\]
and therefore \eqref{eq:proof_thresolventconnectedKsolvI} holds for all $\zeta \in \C$ such that $\vert \zeta - \zeta_0 \vert < \mathrm{dist}(\zeta_0,\sigma(A))$. For such $\zeta$ we have
\begin{equation}\label{eq:proof_thresolventconnectedKsolvII}
\res{\zeta}{A}g = \sum_{n=0}^{\infty} (\zeta-\zeta_0)^n \res{\zeta_0}{A}^{n+1}g\,.
\end{equation}
We claim that $\res{\zeta_0}{A}^n g \in \Kc{A}{g}$ for all $n \in \N_0$. It is clear for the case $n=0$, and for the case $n=1$ by hypothesis. Therefore, we proceed by induction and assume that it is true for a given $n \in \N_0$. We see that $\res{\zeta_0}{A}^{n+1}g =\res{\zeta_0}{A}\res{\zeta_0}{A}^ng$ and as $\res{\zeta_0}{A}^ng \in \Kc{A}{g}$, there exists a polynomial sequence $(p_m)_{m \in \N}$ such that $p_m(A)g \xrightarrow{m \to \infty} \res{\zeta_0}{A}^n g$. By Lemma~\ref{lem:resAk_commute} we see that
\[A^k \res{\zeta_0}{A} g = \res{\zeta_0}{A}A^k g \,, \quad \forall \, k \in \N_0\,,\]
as a consequence of the commutativity of the operator and resolvent $A \res{\zeta_0}{A} \supset \res{\zeta_0}{A}A$. Therefore, $p_m(A) \res{\zeta_0}{A} g = \res{\zeta_0}{A} p_m(A) g$ for all $m \in \N_0$, and from the continuity of the resolvent we obtain
\[\res{\zeta_0}{A}p_m(A) g \xrightarrow{m \to \infty} \res{\zeta_0}{A}^{n+1} g\,.\]
As $A\big(\Kc{A}{g} \cap \mathcal{D}(A) \big) \subset \Kc{A}{g}$, as a consequence of Lemma~\ref{lem:resAk_KAg} we have that additionally $p_m(A) \res{\zeta_0}{A}g \in \Kc{A}{g}$ for all $m \in \N$. Thus $\res{\zeta_0}{A}^{n+1} g \in \Kc{A}{g}$ for all $n \in \N_0$ by induction. So we see that
\[\sum_{n = 0}^\infty (\zeta - \zeta_0)^n \res{\zeta_0}{A}^{n+1} g \in \Kc{A}{g}\,.\]

Now we consider a generic $\zeta' \in \rho(A,\zeta_0)$ and a finite continuous curve $\Gamma:[0,1] \to \C$ contained in the open set $\rho(A,\zeta_0)$ such that $\Gamma(0)=\zeta_0$ and $\Gamma(1) = \zeta'$. Let $\eta$ be the number
\[\eta = \inf_{\zeta \in \Gamma} \mathrm{dist}(\zeta,\sigma(A))\,.\]
$\eta >0$ and is in fact the \emph{minimum} distance from $\Gamma$ to the spectrum $\sigma(A)$ as $\Gamma$ is compact and the distance function is continuous. We now consider open balls with centres $\zeta \in \Gamma$ and radius $\frac{1}{4}\eta$, i.e., the collection of open balls $\{B(\zeta, \frac{1}{4}\eta)\}_{\zeta \in \Gamma}$. Each ball is contained entirely in $\rho(A,\zeta_0)$, and from the compactness of $\Gamma$ we may choose a finite sub-cover of these balls $\{B(\zeta_i, \frac{1}{4}\eta)\}_{i=1}^n$ in particular such that, without loss of generality, $\zeta_0 \in B(\zeta_1,\frac{1}{4}\eta)$, $\zeta' \in B(\zeta_n,\frac{1}{4}\eta)$, and $B(\zeta_i,\frac{1}{4}\eta) \cap B(\zeta_{i+1},\frac{1}{4}\eta) \neq \emptyset$ for all $i=1,\dots,n-1$. For all $\zeta \in B(\zeta_1,\frac{1}{4}\eta)$ we see that $\res{\zeta}{A}g \in \Kc{A}{g}$ by \eqref{eq:proof_thresolventconnectedKsolvII}. As also
\[\vert \zeta_2 - \zeta_0 \vert < \frac{3}{4}\eta < \mathrm{dist}(\zeta_0,\sigma(A))\,,\]
we obtain $\res{\zeta_2}{A} g \in \Kc{A}{g}$ from \eqref{eq:proof_thresolventconnectedKsolvII}. 

Now we consider $\zeta$ such that 
\[\vert \zeta - \zeta_2 \vert < \frac{1}{4}\eta < \mathrm{dist}(\zeta_2,\sigma(A))\,.\]
The following resolvent formulae,
\begin{align}
\res{\zeta}{A} &= \sum_{n=0}^{\infty} (\zeta - \zeta_2)^n \res{\zeta_2}{A}^{n+1}\,, \label{eq:proof_thresolventconnectedKsolvIII}  \\
\res{\zeta}{A}g &= \sum_{n=0}^{\infty} (\zeta-\zeta_2)^n \res{\zeta_2}{A}^{n+1}g\,, \label{eq:proof_thresolventconnectedKsolvIV} 
\end{align}
are valid for all $\vert \zeta - \zeta_2 \vert < \mathrm{dist}(\zeta_2,\sigma(A))$. By following a similar induction argument as above and invoking Lemmas~\ref{lem:resAk_commute} and \ref{lem:resAk_KAg} as was done for $\res{\zeta_0}{A}^n g$, we deduce that $\res{\zeta_2}{A}^n g \in \Kc{A}{g}$ for all $n \in \N_0$. Therefore, $\res{\zeta}{A}g \in \Kc{A}{g}$ for all $\zeta \in B(\zeta_2, \frac{1}{4}\eta)$ as a consequence of \eqref{eq:proof_thresolventconnectedKsolvIV}.

We now consider $\zeta_3$ noting that
\[\vert \zeta_3 - \zeta_2 \vert < \frac{1}{2} \eta < \mathrm{dist}(\zeta_2,\sigma(A))\,.\]
By \eqref{eq:proof_thresolventconnectedKsolvIV} we see that $\res{\zeta_3}{A}g \in \Kc{A}{g}$. Therefore, by following similar arguments as in the case of $\zeta_2$, we see that $\res{\zeta}{A}g \in \Kc{A}{g}$ for all $\zeta \in B(\zeta_3,\frac{1}{4}\eta)$.

We continue as in the previous paragraphs for the other balls, eventually showing that $\res{\zeta}{A}g \in \Kc{A}{g}$ for all $\zeta \in B(\zeta_n,\frac{1}{4}\eta)$, and therefore in particular $\res{\zeta'}{A}g \in \Kc{A}{g}$. As the choice of $\zeta' \in \rho(A,\zeta_0)$ was arbitrary, we conclude the proof.
\end{proof} 

\begin{proof}[Proof of Theorem~\ref{th:poly_resolvent_opnorm}]
Consider any $\zeta > \mathrm{spr} A$. Then it is well-known (see \cite{Kato-perturbation}) that
\begin{equation}\label{eq:res_laurentseries}
\res{\zeta}{A} = -\sum_{n=0}^{\infty} \zeta^{-n-1} A^n \,,
\end{equation}
where the summation is convergent in $\opnorm{\cdot}$. Thus, it is clear that $\res{\zeta}{A} \in \overline{\mathbb{P}(A)}^{\opnorm{\cdot}}$.

Now we take any $\zeta_0 > \mathrm{spr}A$ (the specific choice doesn't matter, as $\rho(A,\zeta_0)$ contains all the points $\zeta > \mathrm{spr}A$), and consider some $\zeta' \in \rho(A,\zeta_0)$. As in the proof of Theorem~\ref{th:resolvent_connected_Ksolv} we consider the finite, continuous simple curve $\Gamma:[0,1] \to \C$ that lies entirely in $\rho(A,\zeta_0)$ such that $\Gamma(0) = \zeta_0$ and $\Gamma(1) = \zeta'$. We cover $\Gamma$ with the finite collection of open balls $\{B(\zeta_i,\frac{1}{4}\eta)\}_{i=1}^n$ in exactly the same way as was done in the proof of Theorem~\ref{th:resolvent_connected_Ksolv}. Furthermore, as $A \in \mathscr{B}(X)$, we have that $A^k \res{\zeta}{A} = \res{\zeta}{A}A^k$ for all $\zeta \in \rho(A)$ and for all $k \in \N_0$ due to the commutativity of the resolvent with $A$. 

It is clear that $\res{\zeta_0}{A} \in \overline{\mathbb{P}(A)}^{\opnorm{\cdot}}$ from \eqref{eq:res_laurentseries}. Suppose that $n \in \N_0$ and, by induction, assume that $\res{\zeta_0}{A}^n \in \overline{\mathbb{P}(A)}^{\opnorm{\cdot}}$. Therefore there exists a sequence of polynomials $(p_m)_{m \in \N}$ such that $p_m(A) \xrightarrow{\opnorm{\cdot}} \res{\zeta_0}{A}^n$. So
\[p_m(A)\res{\zeta_0}{A} \xrightarrow{\opnorm{\cdot}} \res{\zeta_0}{A}^{n+1}\,,\]
and by Lemma~\ref{lem:Pop_invariant} $A^k \res{\zeta_0}{A} \in \overline{\mathbb{P}(A)}^{\opnorm{\cdot}}$ for all $k \in \N_0$, which implies that $p_m(A)\res{\zeta_0}{A} \in \overline{\mathbb{P}(A)}^{\opnorm{\cdot}}$ for all $m \in \N$. Thus $\res{\zeta_0}{A}^{n+1} \in \overline{\mathbb{P}(A)}^{\opnorm{\cdot}}$. So we have shown that $\res{\zeta_0}{A}^n \in \overline{\mathbb{P}(A)}^{\opnorm{\cdot}}$ for all $n \in \N_0$ by induction.

As a result $\res{\zeta}{A} \in \overline{\mathbb{P}(A)}^{\opnorm{\cdot}}$ for all $\zeta$ such that 
\[\vert \zeta - \zeta_0 \vert < \frac{1}{4}\eta < \mathrm{dist}(\zeta_0,\sigma(A))\]
by \eqref{eq:proof_thresolventconnectedKsolvI}, and in particular $\res{\zeta_1}{A} \in \overline{\mathbb{P}(A)}^{\opnorm{\cdot}}$. By a similar induction process as in the previous paragraph, we also find that $\res{\zeta_1}{A}^n \in \overline{\mathbb{P}(A)}^{\opnorm{\cdot}}$ for all $n \in \N_0$. By replacing $\zeta_2$ with $\zeta_1$ in \eqref{eq:proof_thresolventconnectedKsolvIII} and the fact that 
\[\vert \zeta_2 - \zeta_1 \vert < \frac{1}{2}\eta < \mathrm{dist}(\zeta_1,\sigma(A))\,,\]
we obtain the result $\res{\zeta_2}{A} \in \overline{\mathbb{P}(A)}^{\opnorm{\cdot}}$. Again by induction, $\res{\zeta_2}{A}^n \in \overline{\mathbb{P}(A)}^{\opnorm{\cdot}}$ for all $n \in \N_0$. From \eqref{eq:proof_thresolventconnectedKsolvIII} we see that $\res{\zeta}{A} \in \overline{\mathbb{P}(A)}^{\opnorm{\cdot}}$ for all $\zeta \in B(\zeta_2,\frac{1}{4}\eta)$, and In particular we have $\res{\zeta_3}{A} \in \overline{\mathbb{P}(A)}^{\opnorm{\cdot}}$ as $\vert \zeta_3 - \zeta_2 \vert < \mathrm{dist}(\zeta_2,\sigma(A))$. By induction we see that $\res{\zeta_3}{A}^n \in \overline{\mathbb{P}(A)}^{\opnorm{\cdot}}$ for all $n \in \N_0$.

We continue this process and eventually show that $\res{\zeta}{A} \in \overline{\mathbb{P}(A)}^{\opnorm{\cdot}}$ for all $\zeta \in B(\zeta_n,\frac{1}{4}\eta)$, in particular $\res{\zeta'}{A} \in \overline{\mathbb{P}(A)}^{\opnorm{\cdot}}$. As the choice of $\zeta' \in \rho(A,\zeta_0)$ is arbitrary, we conclude the proof.
\end{proof} 

\begin{lemma}\label{lem:resAk_commute}
Let $A:X \to X$ be a closed operator and let $\zeta \in \rho(A)$. Then $\res{\zeta}{A}A^k f = A^k \res{\zeta}{A}f$ for all $k \in \N_0$ and for all $f \in C^{\infty}(A)$.
\end{lemma}
\begin{proof}
Let $\zeta \in \rho(A)$. Due to the fact that (by convention) $A^0 = \1$ and also the fact that $A$ commutes with its resolvent, $\res{\zeta}{A} A \subset A \res{\zeta}{A}$ (\cite{Kato-perturbation}), the lemma is true for the cases $k=0$ and $k=1$.

Proceeding by induction, we assume that the lemma holds for a given $k \in \N_0$, i.e.,
\[\res{\zeta}{A}A^k h = A^k \res{\zeta}{A} h \,,\quad \forall \, h \in C^{\infty}(A)\,.\]
Let $f \in C^{\infty}(A)$, so that we have 
\[\res{\zeta}{A}A^{k+1}f = \res{\zeta}{A}A^k Af\,.\]
It is obvious that $Af \in C^{\infty}(A)$, so that $\res{\zeta}{A}A^k Af = A^k\res{\zeta}{A}Af$, and then from the fact that $A$ and its resolvent commute, we get 
\[\res{\zeta}{A}A^k Af = A^k\res{\zeta}{A}Af = A^{k+1}\res{\zeta}{A}f\,.\]
As the choice of $f \in C^{\infty}(A)$ is arbitrary, this holds for all $f \in C^{\infty}(A)$. We conclude the proof by invoking the induction principle.
\end{proof}

\begin{lemma}\label{lem:resAk_KAg}
Let $A: X \to X$ be a closed operator, let $\zeta \in \rho(A)$, and let $g \in C^\infty(A)$. If
\begin{itemize}
	\item[(i)] $A\big(\Kc{A}{g} \cap \mathcal{D}(A)\big) \subset \Kc{A}{g}$, and
	\item[(ii)] $\res{\zeta}{A}g \in \Kc{A}{g}$,
\end{itemize}
then $A^k \res{\zeta}{A}g \in \Kc{A}{g}$ for all $k \in \N_0$.
\end{lemma}
\begin{proof}
We proceed by induction. It is already obvious for $k =0$ by hypothesis (ii). By Lemma~\ref{lem:resAk_commute} we have that $A^k \res{\zeta}{A}g = \res{\zeta}{A}A^k g$ for all $k \in \N_0$, thus also implying that $A^k \res{\zeta}{A}g \in \mathcal{D}(A)$ for all $k \in \N_0$. Thus, by hypothesis (i) and (ii) we have that the statement is also true for $k = 1$.

Assume that the statement holds true for some $k \in \N_0$. Then,  as $A^k\res{\zeta}{A}g \in \Kc{A}{g} \cap \mathcal{D}(A)$, it is clear that by hypothesis (i) that 
\[A^{k+1}\res{\zeta}{A}g = A A^k \res{\zeta}{A}g \in \Kc{A}{g}\,.\]
By invoking the induction principle we conclude the proof.
\end{proof}

\begin{lemma}\label{lem:Pop_invariant}
Let $A: X \to X$ be a bounded everywhere defined operator, i.e., $A \in \mathscr{B}(X)$. Then for all $\lambda \in \C$ we have the inclusion $(A- \lambda\1)^k \overline{\mathbb{P}(A)}^{\opnorm{\cdot}} \subset \overline{\mathbb{P}(A)}^{\opnorm{\cdot}}$ for all $k \in \N_0$.
\end{lemma}
\begin{proof}
Let $k \in \N_0$. Then $(A - \lambda \1)^k$ is a bounded everywhere defined operator, and therefore defines the continuous mapping $F_k$ on $\mathscr{B}(X)$, $B \xmapsto{F_k} (A - \lambda\1)^k B$, $B \in \mathscr{B}(X)$. As $F_k\big(\mathbb{P}(A)\big) \subset \mathbb{P}(A)$, by the continuity of $F_k$ on $\mathscr{B}(X)$ we conclude.
\end{proof}

\subsection{Isolated points of the spectrum}~

It turns out that the resolvent formalism is quite useful to say something about what happens at isolated points of the spectrum. We use the concept of the \emph{reduced resolvent} to consider the behaviour about these isolated points (see \cite[Ch.~III, Sect.~6]{Kato-perturbation}).

\begin{proposition}\label{prop:isolated_pt}
Let $A:X \to X$ be a bounded everywhere defined operator, and let $\lambda \in \sigma(A)$ be an isolated point of its spectrum. Consider the projection operator 
\[P = -\frac{1}{2\pi i} \int_\Gamma \res{\zeta}{A} \, \rd \zeta\,,\]
for a simple continuous closed curve $\Gamma \subset \rho(A)$ such that only $\lambda$ is in the interior of $\Gamma$. Suppose that there exists $\zeta_0 \in \rho(A)$ such that $\zeta_0$ is in the same connected component of the resolvent as the curve $\Gamma$ and such that $\res{\zeta_0}{A} \in \overline{\mathbb{P}(A)}^{\opnorm{\cdot}}$. Then the reduced resolvent, $\mathcal{R}''(\zeta, A) = \res{\zeta}{A}(\1-P)$ is analytically extendible to $\zeta = \lambda$ and therefore approximable in operator norm by polynomials in $A$, i.e., $\mathcal{R}''(\zeta, A) \in \overline{\mathbb{P}(A)}^{\opnorm{\cdot}}$ for all $\zeta$ inside the closed curve $\Gamma$.

Moreover, if $g \in X$ and $Pg = 0$, then $\mathcal{R}''(\lambda, A)g \in \Kc{A}{g}$ and $f = \mathcal{R}''(\lambda, A)g$ is a solution to the inverse linear problem $(A-\lambda\1)f = g$.
\end{proposition}

\begin{remark}\label{rem:isolated_pt}
We can think of Proposition~\ref{prop:isolated_pt} in simplified terms by considering a compact operator on $X$. It is known that the spectrum of $A$ is discrete, countable, and has unique accumulation point $0$. Moreover, each non-zero $\lambda \in \sigma(A)$ is a finite multiplicity eigenvalue. Requiring that $P g = 0$ is to say that the vector $g$ has no projection onto the eigenspace corresponding to $\lambda$.
\end{remark}

\begin{proof}[Proof of Proposition~\ref{prop:isolated_pt}]
We have the decomposition of the resolvent (see \cite[Sect.~6, Ch.~III]{Kato-perturbation})
\begin{equation}\label{eq:proof_propisolatedptI}
\res{\zeta}{A} = \mathcal{R}'(\zeta, A) + \mathcal{R}''(\zeta,A)\,, \quad \forall\, \zeta \in \rho(A)\,,
\end{equation}
where
\begin{equation}\label{eq:proof_propisolatedptII}
\mathcal{R}'(\zeta, A) = \res{\zeta}{A}P\,, \quad \mathcal{R}''(\zeta, A) = \res{\zeta}{A}(\1 - P)\,, \quad \forall\, \zeta \in \rho(A)\,.
\end{equation}
By \cite[(6.23), Ch.~III]{Kato-perturbation} we have additionally,
\begin{equation}\label{eq:proof_propisolatedptIII}
\mathcal{R}''(\zeta, A) = -\frac{1}{2\pi i} \int_\Gamma \res{\zeta'}{A} \, \frac{\rd \zeta'}{\zeta-\zeta'}\,,
\end{equation}
for all $\zeta$ in the interior of $\Gamma$. Therefore, $\mathcal{R}''(\zeta, A)$ has an analytic continuation and is holomorphic inside $\Gamma$, in particular at $\zeta = \lambda$. Given that by hypothesis $\res{\zeta_0}{A} \in \overline{\mathbb{P}(A)}^{\opnorm{\cdot}}$, by Theorem~\ref{th:poly_resolvent_opnorm} we know that $\res{\zeta'}{A} \in \overline{\mathbb{P}(A)}^{\opnorm{\cdot}}$ for all $\zeta' \in \Gamma$. Therefore, given any $\zeta$ in the interior of $\Gamma$ it is immediate that $\frac{1}{\zeta - \zeta'} \res{\zeta'}{A} \in \overline{\mathbb{P}(A)}^{\opnorm{\cdot}}$ for all $\zeta' \in \Gamma$. As the integral in \eqref{eq:proof_propisolatedptIII} is actually the limit in $\opnorm{\cdot}$ of a Riemann sum, we have that each Riemann sum is in $\overline{\mathbb{P}(A)}^{\opnorm{\cdot}}$ and thus the integral itself is also in $\overline{\mathbb{P}(A)}^{\opnorm{\cdot}}$. As such, we have established that $\mathcal{R}''(\zeta, A) \in \overline{\mathbb{P}(A)}^{\opnorm{\cdot}}$ for all $\zeta$ inside $\Gamma$, and in particular for $\zeta = \lambda$.

Suppose that $g \in X$ is such that $Pg = 0$ (and thus $(\1 - P)g = g$). It is clear by the analytic continuation of $\mathcal{R}''(\zeta, A)$ discussed above that
\[\mathcal{R}''(\lambda, A) = \lim_{\zeta \to \lambda} \mathcal{R}''(\zeta, A) = \lim_{\zeta \to \lambda} \res{\zeta}{A}(\1 - P)\,,\]
in $\opnorm{\cdot}$. So, $\mathcal{R}''(\lambda, A) g = \lim_{\zeta \to \lambda} \mathcal{R}''(\zeta, A)g$.

Consider the inverse linear problem $(A - \lambda\1)f = g$. We let $f^\circ = \mathcal{R}''(\lambda, A)g$, from which
\begin{align*}
(A - \lambda\1)\mathcal{R}''(\lambda, A) g &= \lim_{\zeta \to \lambda} (A - \lambda \1)\res{\zeta}{A}g\\
 &= \lim_{\zeta \to \lambda} \big(g + (\zeta - \lambda)\res{\zeta}{A}g\big)\\
 &= g + \lim_{\zeta\to \lambda}(\zeta-\lambda)\res{\zeta}{A}g\\
 &= g + \lim_{\zeta\to \lambda}(\zeta-\lambda) \cdot \lim_{\zeta\to \lambda}\res{\zeta}{A}g\\
 &= g + \lim_{\zeta\to \lambda}(\zeta-\lambda) \cdot \mathcal{R}''(\lambda,  A)g \\
 &= g\,,
\end{align*}
owing to the continuity of $A$ and the fact that $\res{\zeta}{A}(\1-P)g = \res{\zeta}{A}g$. This shows that $f^\circ$ is indeed a solution to the inverse linear problem $(A-\lambda\1)f = g$. Moreover, as $\mathcal{R}''(\lambda, A) \in \overline{\mathbb{P}(A)}^{\opnorm{\cdot}}$, it is clear that $\mathcal{R}''(\lambda, A)g \in \Kc{A}{g}$, i.e., $f^\circ$ is a Krylov solution.
\end{proof}

We show with the following example that though Proposition~\ref{prop:isolated_pt} provides some nice properties of the reduced resolvent that lend themselves to Krylov solvability, the conditions therein are not necessary to guarantee Krylov solvability.

\begin{example}\label{eg:Volterra_op}
We consider the Volterra operator on $X = L^2[0,1]$, $V: X \to X$, $f \mapsto \int_0^x f(y) \, \rd y$. $V$ is a quasi-nilpotent operator, whence it has spectral radius $0$ and therefore the isolated point $\lambda=0$ is the only point in the spectrum (\cite{Kato-perturbation}). The projection operator $P$ is given by 
\[P = -\frac{1}{2\pi i}\int_\Gamma \res{\zeta}{V} \, \rd \zeta\,.\] 
Moreover $V$ has no eigenvalues, and it is also known \cite{CMN-2018_Krylov-solvability-bdd} that for any vector $g$ a polynomial in $x$,
\[\Kc{V}{x^k} = L^2[0,1]\,, \quad \forall\, k \in \N_0\,.\]

The vector $g(x) = x$ is in the range of the operator $V$ thereby ensuring that the inverse linear problem $Vf = g$ is solvable, and moreover Krylov solvable, yet we show that in this case $Pg \neq 0$. Indeed, the analytical form the resolvent is given by
\[\big(\res{\zeta}{V}h\big)(x) = -\zeta^{-1} h(x) - \zeta^{-2}\int_0^x \exp\bigg(\frac{x-y}{\zeta}\bigg)h(y) \, \rd y \,,\]
for $\zeta \in \C \setminus \{0\}\,,\,\, h \in L^2[0,1]$. As we have that
\[Pg = -\frac{1}{2\pi i} \int_\Gamma \res{\zeta}{V}g \,\rd \zeta \,, \quad \forall \, \zeta \in \C \setminus \{0\}\,,\]
for $\Gamma$ a simple closed curve with $0$ in its interior, for $g(x) = x$ we see that
\begin{align*}
Pg &= \frac{1}{2\pi i} \int_\Gamma \zeta^{-1}x + \zeta^{-2} \int_0^x \exp\bigg(\frac{x-y}{\zeta}\bigg) y \, \rd y \, \rd \zeta \\
 &= x + \frac{1}{2\pi i} \int_\Gamma \zeta^{-2} \int_0^x \exp\bigg(\frac{x-y}{\zeta}\bigg) y \, \rd y \, \rd \zeta\,.
\end{align*}
Furthermore,
\[\int_0^x \exp\bigg(\frac{x-y}{\zeta}\bigg) y \, \rd y = -\zeta^2 \bigg(1 + \frac{x}{\zeta}\bigg) + \zeta^2 \exp\bigg(\frac{x}{\zeta}\bigg)\,,\]
whence
\begin{align*}
Pg &= x + \frac{1}{2 \pi i} \bigg(\int_\Gamma -1 - \frac{x}{\zeta} + \exp\bigg(\frac{x}{\zeta}\bigg)\, \rd \zeta \bigg)\\
 &= x - x +  \frac{1}{2\pi i}\int_\Gamma  \exp\bigg(\frac{x}{\zeta}\bigg) \, \rd \zeta\,,
\end{align*}
and by using the Laurent series expansion for $\exp(x/\zeta)$ and the residue formula (see \cite{lang1999-complex}), we obtain
\[(Pg)(x) = \frac{1}{2\pi i}\int_\Gamma  \exp\bigg(\frac{x}{\zeta}\bigg) \, \rd \zeta = x \,,\]
whence $Pg = g$.

This clearly demonstrates that the conditions in Proposition~\ref{prop:isolated_pt}, while informative and intuitive, are not necessary conditions to guarantee the Krylov solvability in such cases of isolated singularities.
\end{example}

This is further highlighted in the following proposition and its following remarks.

\begin{proposition}\label{prop:separatedspectrum_projection}
Let $A : X \to X$ be a closed operator and let $g \in C^{\infty}(A)$. Suppose that $\Gamma \subset \rho(A)$ is a continuous closed, simple curve that separates the spectrum into two parts: $\sigma'(A)$ the part of the spectrum inside $\Gamma$ and $\sigma''(A)$ the part of the spectrum outside $\Gamma$. Suppose further that
\begin{itemize}
	\item[(i)] there exists $\zeta_0 \in \rho(A)$ such that $\res{\zeta_0}{A}g \in \Kc{A}{g}$, and moreover that
	\item[(ii)] $\zeta_0$ is in the connected part of $\rho(A)$ containing $\Gamma$, and
	\item[(iii)] $A \big(\Kc{A}{g} \cap \mathcal{D}(A)\big) \subset \Kc{A}{g}$.
\end{itemize}

Then,
\[Pg = \bigg(-\frac{1}{2\pi i} \int_\Gamma \res{\zeta}{A} \, \rd \zeta \bigg) g \in \Kc{A}{g}\,.\]
Moreover, if $A$ is bounded everywhere defined and $\zeta_0$ is such that $\res{\zeta_0}{A} \in \overline{\mathbb{P}(A)}^{\opnorm{\cdot}}$, then $P \in \overline{\mathbb{P}(A)}^{\opnorm{\cdot}}$.
\end{proposition}

\begin{remark}\label{rem:purelydiscrete_projection}
If in Proposition~\ref{prop:separatedspectrum_projection} we consider $A$ with a purely discrete spectrum, then there is only a single connected part of $\rho(A)$. In this case we may remove the need for assumption (ii) in the hypothesis of the proposition.

We also note that for the special case of $A$ being bounded and everywhere defined, hypothesis (iii) is guaranteed by the continuity of $A$.
\end{remark}

\begin{remark}\label{rem:quasinilpotent_P}
We remark that in the case of a (quasi-)nilpotent operator $A$ (i.e., $\mathrm{spr} A = 0$), the conditions (i), (ii), and (iii) of Proposition~\ref{prop:separatedspectrum_projection} are always satisfied for all $\zeta_0 \neq 0$. Yet we see that, although the proposition guarantees that $P \in \overline{\mathbb{P}(A)}^{\opnorm{\cdot}}$, by using the Laurent series for the resolvent in \eqref{eq:res_laurentseries} in the proof of Theorem~\ref{th:poly_resolvent_opnorm}, we obtain
\[\int_\Gamma \res{\zeta}{A} \, \rd \zeta = -2\pi i \1\,,\]
by an application of the residue formula for complex integrals. Therefore, $P = \1$ for (quasi-)nilpotent operators.
\end{remark}

\begin{proof}[Proof of Proposition~\ref{prop:separatedspectrum_projection}]
By Theorem~\ref{th:resolvent_connected_Ksolv}, for all $\zeta \in \Gamma$ we have $\res{\zeta}{A}g \in \Kc{A}{g}$. As the integral for $P$ is actually an operator norm convergent Riemann summation, we have that $Pg \in \Kc{A}{g}$. Indeed,
\begin{equation}\label{eq:proof_separatedspectrumprojection}
\lim_{n \to \infty} \opnorm{\sum_{i=1}^n \res{\zeta_i}{A} \Delta\zeta_i - \int_\Gamma \res{\zeta}{A} \, \rd \zeta} = 0\,, \tag{\ddag}
\end{equation}
for $\{\zeta_i\}_{i=1}^n$ a partition of $\Gamma$, whence
\begin{align*}
\lim_{n \to \infty} & \norm{\sum_{i=1}^n \res{\zeta_i}{A} \Delta\zeta_i g - \bigg(\int_\Gamma \res{\zeta}{A} \, \rd \zeta\bigg)g}{X} \\ 
& \leqslant \lim_{n \to \infty} \opnorm{\sum_{i=1}^n \res{\zeta_i}{A} \Delta\zeta_i - \int_\Gamma \res{\zeta}{A} \, \rd \zeta} \norm{g}{X}\\
 & = 0\,.
\end{align*}
As the Riemann summation $\sum_{i=1}^n \res{\zeta_i}{A} \Delta\zeta_i g \in \Kc{A}{g}$ for all $n \in \N$ as $\res{\zeta_i}{A} \Delta\zeta g \in \Kc{A}{g}$ for all $\zeta_i$, the result follows.

If, additionally we assume that $A$ is bounded and that $\res{\zeta_0}{A} \in \overline{\mathbb{P}(A)}^{\opnorm{\cdot}}$, then by \eqref{eq:proof_separatedspectrumprojection} and Theorem~\ref{th:poly_resolvent_opnorm} the result follows.
\end{proof}

\begin{corollary}\label{cor:compact_projection}
Let $A:X \to X$ be a compact operator on the Banach space $X$ and let $\lambda \in \sigma(A)$ be any non-zero eigenvalue. Then $P_\lambda \in  \overline{\mathbb{P}(A)}^{\opnorm{\cdot}}$ and also the nilpotent operator $D_\lambda \in  \overline{\mathbb{P}(A)}^{\opnorm{\cdot}}$, where
\[P_\lambda =  -\frac{1}{2\pi i} \int_{\Gamma_\lambda} \res{\zeta}{A} \, \rd \zeta\,, \quad D_\lambda = (A - \lambda\1)P_\lambda\,,\]
for $\Gamma_\lambda \subset \rho(A)$ a continuous simple closed curve with only $\lambda$ in its interior.
\end{corollary}
\begin{proof}
We note that for $\zeta_0 > \mathrm{spr}A$ we have that $\res{\zeta_0}{A} \in \overline{\mathbb{P}(A)}^{\opnorm{\cdot}}$ by Theorem~\ref{th:poly_resolvent_opnorm}. Moreover, by Remark~\ref{rem:purelydiscrete_projection} the resolvent set $\rho(A)$ is a single connected set. Thus, by application of Proposition~\ref{prop:separatedspectrum_projection} we have that $P_\lambda \in \overline{\mathbb{P}(A)}^{\opnorm{\cdot}}$. The fact that $D_\lambda \in \overline{\mathbb{P}(A)}^{\opnorm{\cdot}}$ then follows from Lemma~\ref{lem:Pop_invariant}.
\end{proof}

\bibliographystyle{siam}
\def\cprime{$'$}

\end{document}